\documentclass[12pt,twoside]{extarticle}
\usepackage[english]{babel}
\usepackage[utf8]{inputenc}
\usepackage{geometry}
\usepackage{fancyhdr}
\usepackage{lipsum}
\usepackage{titlesec}
\usepackage{amsmath}
\usepackage{amssymb}
\usepackage{amsthm}
\usepackage{enumitem}
\usepackage{mathtools}
\usepackage{bigints}
\usepackage{mathrsfs}
\usepackage[]{enumitem}
\usepackage[T1]{fontenc}
\usepackage{xpatch}
\usepackage{theoremref}
\usepackage{comment}
\usepackage{varwidth}
\usepackage{setspace}
\usepackage[dvipsnames]{xcolor}
\usepackage[utf8]{inputenc}
\usepackage[titles]{tocloft}
\usepackage[export]{adjustbox}
\usepackage{graphicx} 
\usepackage{setspace}
\usepackage{microtype}
\usepackage{upgreek}
\usepackage{pgf}
\usepackage{tikz}
\usepackage{caption}
\usepackage{subcaption}
\usepackage{float}
\usepackage[export]{adjustbox}
\usepackage{lipsum}
\usepackage{esint}
\usepackage[misc]{ifsym}
\usepackage{standalone}
\usepgflibrary{arrows}
\usepackage{bbm}
\usepackage{scalerel} 
\usepackage{calrsfs}
\usepackage{esint}
\usepackage{yfonts}
\usepackage{blindtext}
\blindmathtrue
\usepackage{graphicx}
\usepackage[toc,page]{appendix}
\usepackage{authblk}
\usepackage{mathabx}
\usepackage{mlmodern}
\setenumerate{itemsep=0pt,topsep=4pt,parsep=2pt,partopsep=5pt}
\setitemize{itemsep=0pt,topsep=4pt,parsep=2pt,partopsep=5pt}
\renewcommand{\thesection}{\arabic{section}}

\titleformat{\chapter}
    {\Huge \bfseries \centering}
   {}
   {5pt}
    {}
    []
\titleformat{\section}
    {\large \bfseries \centering} 
    {\thesection.} 
    {5pt} 
    {} 
\titleformat{\subsection}
    {\bfseries}
    {\thesubsection.}
    {5pt}
    {}
\titleformat{\subsubsection}
    [runin]
    {\bfseries}
    {\thesubsubsection}
    {5pt}
    {\adddoubleperiod}

    \DeclareMathAlphabet{\pazocal}{OMS}{zplm}{m}{n}
    \DeclareMathAlphabet{\mathcal}{OMS}{cmsy}{m}{n}

\makeatletter
\def\thanks#1{\protected@xdef\@thanks{\@thanks
        \protect\footnotetext{#1}}}
\makeatother

\newcommand{\Rnm}{\ensuremath{{\mathbb{R}^{n+m}}}}
\newcommand{\Rn}{\ensuremath{{\mathbb{R}^{n}}}}
\newcommand{\R}{\ensuremath{{\mathbb{R}}}}
\newcommand{\Rp}{\ensuremath{{\mathbb{R}^{+}}}}
\newcommand{\tc}{\ensuremath{\mathcal{TC}}}
\newcommand{\cd}{\ensuremath{\mathcal{CD}}}

\newcommand{\stc}[1]{\ensuremath{\mathcal{STC}({#1})}}

\newcommand{\calM}{\mathbf{M}}
\newcommand{\Hn}{{\mathscr{H}^{n}}}
\newcommand{\B}{{\mathscr{B}}}
\newcommand{\scrH}{{\mathscr{H}}}
\newcommand{\scrL}{{\mathscr{L}}}
\newcommand{\Hp}{{\mathscr{H}^{n-1}}}
\newcommand{\Ln}{{\mathscr{L}^{n}}}

\newcommand{\M}{\mathfrak{M}}
\newcommand{\m}{\mathfrak{m}}

\renewcommand{\Tilde}{\widetilde}

\renewcommand{\div}{\text{\normalfont div}}

\newcommand{\norm}[1]{\lVert #1 \rVert}

\renewcommand{\ae}{\text{\normalfont a.e.}}

\newcommand{\Lip}{\text{\normalfont Lip}}

\newcommand{\loc}{\text{\normalfont loc}}
\newcommand{\ps}{\text{\footnotesize\normalfont PS}}
\newcommand{\I}{\text{\footnotesize\normalfont I}}
\renewcommand{\S}{\text{\footnotesize\normalfont S}}
\newcommand{\at}{\text{\footnotesize\normalfont TA}}
\newcommand{\gn}{\text{\footnotesize\normalfont GN}}
\newcommand{\ic}{\text{\footnotesize \normalfont IC}}
\newcommand{\be}{\text{\footnotesize \normalfont B}}

\newcommand{\g}{\text{\footnotesize\normalfont G}}
\newcommand{\egn}{\text{\footnotesize\normalfont EGN}}
\newcommand{\ls}{\text{\footnotesize\normalfont LS}}
\newcommand{\textn}[1]{\text{\normalfont #1}}
\renewcommand{\theta}{\vartheta}
\renewcommand{\phi}{\varphi}
\renewcommand{\epsilon}{\varepsilon}

\DeclareSymbolFont{extraup}{U}{zavm}{m}{n}
\DeclareMathSymbol{\varheart}{\mathalpha}{extraup}{86}
\DeclareMathSymbol{\vardiamond}{\mathalpha}{extraup}{87}
\newcommand{\Lcorner}{\mathrel{\makebox[7pt][c]{\rule{.4pt}{5.75pt}\rule{4.5pt}{.4pt}}}}
\renewcommand{\llcorner}{{\Lcorner}}
\newtheoremstyle{defbfnote}
    {15pt} 
    {15pt} 
    {} 
    {} 
    {\bfseries} 
    {.} 
    {.5 em} 
    {\thmname{#1}\thmnumber{ #2}\thmnote{ (#3)}}
\newtheoremstyle{bfnote}
    {15pt} 
    {15pt} 
    {\itshape} 
    {} 
    {\bfseries} 
    {.} 
    {.5 em} 
    {\thmname{#1}\thmnumber{ #2}\thmnote{ (#3)}}
\newtheoremstyle{scesempio}
    {10pt} 
    {10pt} 
    {} 
    {} 
    {\bfseries} 
    {.} 
    {.5 em} 
    {\thmname{#1}\thmnumber{ #2}\thmnote{ (#3)}}
\newtheoremstyle{bfesercizio}
    {20pt} 
    {10pt} 
    {\itshape} 
    {} 
    {\bfseries} 
    {.} 
    {.5 em} 
    {\thmname{#1}\thmnumber{ #2}\thmnote{ (#3)}}
\newtheoremstyle{scsvolgimento}
    {10pt} 
    {30pt} 
    {} 
    {} 
    {\sc} 
    {.} 
    {.5 em} 
    {\thmname{#1}}

\theoremstyle{bfnote}
\newtheorem{theorem}{Theorem}

\newtheorem{corollary}[theorem]{Corollary}
\newtheorem{lemma}[theorem]{Lemma}
\newtheorem{proposition}[theorem]{Proposition}
\theoremstyle{defbfnote}
\newtheorem*{definition}{Definition}
\newtheorem*{remark}{Remark}
\theoremstyle{bfesercizio}

\theoremstyle{scsvolgimento}

\theoremstyle{scesempio}
\newtheorem{example}[theorem]{Example}
\theoremstyle{bfnote}

\title{{\textbf{\textbf{\large \bfseries\uppercase{P\'olya-Szeg\H{o} inequalities on submanifolds with small total mean curvature}}}}}
\author{\sc pietro aldrigo and zolt\'{a}n m. balogh 
\thanks{The authors are supported by the Swiss National Science Foundation, grant number 200021-228012. }
}
\date{}

\begin{document}

\maketitle

\begin{abstract}
    We establish P\'olya-Szeg\H{o}-type inequalities (PSIs) for Sobolev-functions defined on a regular $n$-dimensional submanifold $\Sigma$ (possibly with boundary) of a $(n+m)$-dimensional Euclidean space, under explicit upper bounds of the total mean curvature.
    The $p$-Sobolev and Gagliardo-Nirenberg inequalities, as well as the spectral gap in $W^{1,p}_0(\Sigma)$ are derived as corollaries. Using these PSIs, we prove a sharp $p$-Log-Sobolev inequality for minimal submanifolds in codimension one and two.
    The asymptotic sharpness of both the multiplicative constant appearing in PSIs and the assumption on the total mean curvature bound as $n\to \infty$ is provided.
    A second equivalent version of our PSIs is presented in the appendix of this paper, introducing the notion of model space $(\Rp,\m_{n,K})$ of dimension $n$ and total mean curvature bounded by $K$.
\end{abstract}


\section{Introduction}
    The P\'olya-Szeg\H{o} inequalities (PSIs) serve as a main tool to establish sharp Sobolev, Gagliardo-Nirenberg, and spectral gap inequalities. PSIs rely on symmetrization techniques and the sharp isoperimetric inequality in a Euclidean setting.  This approach uses the Schwartz rearrangement \cite{kesavan2006symmetrization,lieb2001analysis}, a transformation that maps a function to a radially symmetric counterpart while preserving the measure of its upper-level sets. This technique enables the derivation of functional inequalities in Euclidean settings by exploiting the isoperimetric properties of Euclidean open balls \cite{de1958sulla,evans2018measure,lieb2001analysis}.  However, extending this methodology to prove analogous inequalities on submanifolds presents a more intricate challenge, primarily due to problems that are not present in flat Euclidean spaces.  The current paper addresses this by establishing PSIs for Sobolev functions defined on $n$-dimensional submanifolds embedded within a higher-dimensional Euclidean space $\Rnm$.

    The general method for proving PSIs in the Euclidean context was established in \cite{polya1945inequalities}. More recently, similar arguments have been shown to be effective in the context of $\cd(K,N)$-spaces (see \cite{balogh2023sharp, Mondino_2020}). In this work, we apply this method in the framework of Euclidean submanifolds.

    A key aspect of our analysis lies in the role played by the submanifold's total mean curvature. For a smooth $n$-dimensional submanifold $\Sigma \subseteq \Rnm$ (possibly with boundary), the total mean curvature is defined as 
    \begin{align*}
        \tc(\Sigma):= \norm{H}_{L^n(\Sigma)}.
    \end{align*}
    In this expression, $H:\Sigma\to \Rnm$ represents the mean curvature vector field, which is defined as  the trace of the second fundamental form of the submanifold $\Sigma$ \cite{do1992riemannian,lee2006riemannian}, or equivalently, as the first variation of the area functional of the submanifold \cite{simon1984lectures}.  The total mean curvature, as detailed further in section \ref{sec:ProofThPs}, serves as a measure of the cumulative effect of curvature across the entire submanifold. We then proceed to define the class of $n$-dimensional submanifolds of $\Rnm$ with total mean curvature smaller than $K$, where $K>0$ is a fixed constant, as
    \begin{align*}
        \stc{K;n+m,n}:=\left\{\Sigma \subseteq \Rnm\,\,n\textn{-dimensional submanifold }: \tc(\Sigma)\leq K\right\}.
    \end{align*}
    When the dimension $n$ and codimension $m$ are given, we abbreviate this notation and write $\stc{K}$.

    Michael and Simon in \cite{michael1973sobolev} established the existence of a constant $C_n$ that depends solely on $n$ (and not on $m$), such that the isoperimetric inequality
    \begin{align}\label{IsoIneq}
        (\Hn(\Sigma))^{\frac{n-1}{n}}\leq C_n\left(\Hp(\partial \Sigma) + \int_\Sigma |H|\right)
    \end{align}
    hold for every $n$-dimensional submanifold $\Sigma\subseteq \Rnm$. Moreover, they showed that one can choose $C_n \leq 5^n/\omega_n^{1/n}$. Let us define, for each integer $n\geq 2$, the \emph{``isoperimetric constant of $n$-dimensional submanifolds''}, $\ic(n)$, as the infimum of all constants $C_n$ realizing \eqref{IsoIneq}  for every $n$-dimensional submanifold $\Sigma\subseteq \Rnm$ and for every $m\geq 1$. 

    The primary result of this paper is the following theorem.
\begin{theorem}[P\'olya-Szeg\H{o} inequality]\label{th:PS}
    Let $\Sigma\subseteq\Rnm$ be a $n$-dimensional submanifold such that $\Sigma \in \stc{K}$ for  $K<1/\ic(n)$. If $u\in W^{1,p}_0(\Sigma;\Rp)$ for some $1\leq p<\infty$ and $u^*:\Rn\to\Rp$ is its Schwartz rearrangement with respect to $\Ln$, then $u^*\in W^{1,p}(\Rn;\Rp)$ and
        \begin{align}\label{claim:PS}
            \norm{\nabla u^*}_{L^p(\Rn)}\leq \ps(n,K)\norm{\nabla_\Sigma u}_{L^p(\Sigma)},\quad \ps(n,K):=\frac{\ic(n)}{1-\ic(n)\,K}n\omega_n^{1/n}
        \end{align}
\end{theorem}

The constraint on the total mean curvature turns out to be a fundamental requirement for the validity of \eqref{claim:PS}, as demonstrated in Example \ref{ex:counter} of Section \ref{ch:Sharpness} and the subsequent remarks. This example also shows the asymptotic sharpness of the bound on $K$ as the dimension $n$ increases to infinity. More precisely, we consider  $\Sigma={\mathbb{S}^2}$,  the two-dimensional sphere in $\R^3$ and note that by Proposition \ref{prop:isoStc}), ${\mathbb{S}^2}\not\in\stc{K}$ for any $K<1/\ic(2)$. In  Example \ref{ex:counter} we show, that  there exists a family of functions $\left\{u_\lambda \in \Lip({\mathbb{S}^2};\Rp) : \lambda \in[1,\infty)\right\}$ such that for every $p>1$ and every $N\in \mathbb{N}$, there exists $\overline{\lambda}(N,p)\geq 1$ such that 
    \begin{align}\label{conterexample:claim}
        \int_{\R^2} |\nabla u^*_\lambda|^p  >  N \int_{\mathbb{S}^2}\left(|\nabla_{\mathbb{S}^2} u_\lambda|^p + |H|^pu_\lambda^p\right) \quad \forall \lambda \geq \overline{\lambda}(N,p).
    \end{align}

Theorem \ref{th:PS} yields several Euclidean-type functional inequalities as corollaries, for submanifolds $\Sigma\in \stc{K}$ with $K < 1/\ic(n)$.  These inequalities, derived from \eqref{claim:PS}, extend known first-order integral inequalities from Euclidean space to this class of submanifolds. 
In particular, in Corollary \ref{cor:MonPrinc} we prove that if a function $v\in \Lip_c(\Rn)$ satisfies 
    \begin{align}
        L\left(\int_\Rn f(v),\int_\Rn g(|\nabla v|)\right) \leq \Lambda\left(\int_\Rn \phi(v),\,\int_\Rn \psi(|\nabla v|)\right),
    \end{align}
    where $f,g,\phi,\psi,L,\Lambda$ possess suitable monotonicity assumptions, and $\Sigma \in \stc{K}$ for some $K<1/\ic(n)$, then every function $u\in \Lip_c(\Sigma)$ whose Schwartz rearrangement $u^*$ coincides with $v$ satisfies
    \begin{align}
        L\left(\int_\Sigma f(u),\int_\Sigma g(\ps(n,K)\nabla_\Sigma u)\right) \leq \Lambda\left(\int_\Sigma \phi(u),\,\int_\Sigma \psi(\ps(n,K) |\nabla_\Sigma u|)\right).
    \end{align}

This corollary implies, as special cases, all classical Sobolev, Gagliardo-Nirenberg and spectral gap inequalities.

It is worth noting that no assumptions on Ricci curvature are required in Theorem \ref{th:PS}; hence, the results apply to Riemannian submanifolds with non-positive Ricci curvature as well, including minimal surfaces. Therefore, these results develop in a different complementary direction compared to \cite{balogh2023sharp, impera2025isoperimetric, Mondino_2020, mondino2021talenti, nobili2022rigidity}.

Brendle in \cite{brendle2021isoperimetric}  established that the sharp isoperimetric constant is equal to the Euclidean one $n^{-1}\omega_n^{-1/n}$, for codimension $m \in \{1, 2\}$ (see also \cite{BrendleEichmair+2023+1+10,brendle2024isoperimetric}). For general $m\geq 1$,  he provided the isoperimetric inequality: 
\begin{gather*}
    (\Hn(\Omega))^{\frac{n-1}{n}}\leq \be(n,m) \left(\Hp(\partial \Omega) + \int_\Omega |H|\right), \\\quad \be(n,m):=
    \begin{cases}
    \frac{1}{n\omega_n^{1/n}}&,\textn{ if } m\in \{1,2\}\\
    \min\left\{\frac{1}{n}\left(\frac{m\omega_m}{(n+m)\omega_{n+m}}\right)^{\frac{1}{n}},\frac{5^n}{\omega_n^{1/n}}\right\}&,\textn{ if }m\geq 3
    \end{cases}
\end{gather*}
for every domain $\Omega \subseteq \Sigma$. Using the explicit constant $\be(n,m)$ instead of $\ic(n)$ in Theorem \ref{th:PS} allows for explicit quantitative estimates of the constant involved in the aforementioned inequalities, although this introduces the dependence on the codimension $m$.
In particular, replacing $\ic(n)$ with $\be(n,m)$, the constant $\ps$ of Theorem \ref{th:PS} becomes 1, for $m\in\{1,2\}$ and $K=0$. 
This allows us to prove the \emph{sharp} version of Sobolev, Gagliardo-Nirenberg and spectral gap inequalities. In particular, as a consequence of the sharp Gagliardo-Nirenberg, we obtain the \emph{sharp} Logarithmic Sobolev inequalities for minimal submanifolds in codimension $m\in\{1,2\}$.

\begin{corollary}[Sharp $p$-Log-Sobolev inequality for minimal submanifolds]\label{th:SLS}
    Let $\Sigma\subseteq \Rnm$ be a $n$-dimensional minimal submanifold, where $n\geq 2$ and $m\in\{1,2\}$. Fix $p\in(1,n)$ and $u\in W^{1,p}_0(\Sigma)$ such that $\norm{u}_{L^p}=1$. Then
    \begin{align}\label{claim:SLS}
        \int_\Sigma |u|^p\ln|u|^p \leq \frac{n}{p}\ln\left(\ls(n,p)\int_\Sigma |\nabla_\Sigma u|^p\right),
    \end{align}
    with
    \begin{align*}
        \ls(n,p):= \frac{p}{n\pi^{p/2}}\left(\frac{p-1}{e}\right)^{p-1}\left(\frac{\Gamma\left(\frac{n}{2}+1\right)}{\Gamma\left(n\frac{p-1}{p}+1\right)}\right)^{\frac{p}{n}}.
    \end{align*}
    The inequality \eqref{claim:SLS} is sharp. In particular, equality is obtained in the case $\Sigma=\Rn$ and $u(x)=Ce^{\frac{1}{s}|x-\overline{x}|^{p/(p-1)}}$ for any $s>0$ $\overline{x}\in \Rn$, where $C>0$ is the normalizing constant. 
\end{corollary}
Compare this Corollary \ref{th:SLS} with Theorem 1.2 of \cite{balogh2024logarithmic}. In the latter, the second author and Krist\'aly show a $p$-Log-Sobolev inequality with a different constant for minimal submanifolds in any codimension and, more importantly, for any $p\geq 2$ using optimal mass transport. 

Observe that \eqref{claim:PS}, \eqref{monPrinc:Claim} and \eqref{claim:SLS} are all of \emph{``Euclidean type''}, in the sense that no additional terms involving the integral of the mean curvature appear. Instead, only the upper bound $K$ of the total mean curvature of $\Sigma$ influences the \emph{multiplicative} constant $\ps$ in \eqref{claim:PS}. Compare this with the inequalities obtained in \cite{balogh2024logarithmic,BrendleEichmair+2023+1+10,brendle2024isoperimetric,michael1973sobolev}.

We now present the structure of this paper. 

In Section \ref{sec:Prelim}, we recall the definition of Schwartz rearrangement for general measure spaces and the basic properties that will be used throughout this work. Particular attention is devoted to the properties of the Schwartz rearrangement for functions defined on a $n$-dimensional submanifold. The proof of Theorem \ref{th:PS} is presented in Section \ref{sec:ProofThPs}. Section \ref{sec:Cor} contains several consequences of Theorem \ref{th:PS}, such as a monotonicity principle for integral inequalities involving the first-order derivatives for functions defined on submanifolds (Corollary \ref{cor:MonPrinc}), which in turn implies the $p$-Sobolev inequality (Corollary \ref{cor:pSobStc}), Gagliardo-Nirenberg inequality (Corollary \ref{cor:GN_Stc}) and the Spectral Gap (Corollary \ref{cor:SGStc}). At the end of this section, we provide explicit estimates of the constants involved in these inequalities replacing $\ic(n)$ with $\be(n,m)$ and prove the asymptotic sharpness of the constant $\ps$ in codimension $m=1$ and $m=2$ as $n\to\infty$. Using these final remarks, we are able to prove Corollary \ref{th:SLS}. In Section \ref{ch:Sharpness} we prove, through an explicit example, that Theorem \ref{th:PS} does not hold when $\Sigma$ is a compact submanifold without boundary and, as a consequence of this fact, that the total mean curvature bound is asymptotically sharp as the dimension of the manifolds $n$ grows to infinity. Finally, in the appendix, we consider the concept of model space in the spirit of Milman \cite{milman2011sharp}. This involves constructing a one-dimensional space $(\Rp,\m_{n,K})$, where $\m_{n,K}$ is a measure dependent solely on the dimension $n$ and the bound $K$ on the total mean curvature, which captures the essential measure-theoretical properties of a $n$-dimensional submanifold with total curvature bounded by $K$. This allows us to reformulate our main theorem within this simpler setting, providing an equivalent but arguably more tractable version of the P\'olya-Szeg\H{o} inequalities.

\section{Preliminary results and notation}\label{sec:Prelim}
\subsection{The Schwartz rearrangement}\label{subsection:Schwartz}

 For every $N\geq 2$, $\calM(\R^N)$ denotes the family of Radon measures $\m:\B(\R^N)\to[0,\infty]$ that are absolutely continuous with respect to the Lebesgue measure $\scrL^N$ on $\R^N$, with positive density and $\m(\R^N)=\infty$. In the case $N=1$, we no longer consider measures defined on the whole real line, but only on the non-negative half-line $\Rp:=[0,\infty)$. That is, we denote by $\calM(\Rp)$ the family of Borel measures $\m:\B(\Rp)\to[0,\infty]$ that are absolutely continuous with respect to the restriction of the Lebesgue measure $\scrL\llcorner\Rp$, with positive density and $\m(\Rp)=\infty$. With a slight abuse of notation, for every integer $N \geq 1$, we write $\calM(\mathbb{R}^N)$, where for $N=1$, $\calM(\mathbb{R}^+)$ is understood.

{ \renewcommand{\Rn}{{\mathbb{R}^N}}
    \renewcommand{\Ln}{{\mathscr{L}^N}}

\begin{definition}[Schwartz rearrangement]
    Let $(X,\M)$ be a measure space, $N\geq 1$ and integer and $\m\in \calM(\Rn)$. A $\M$-measurable function $f:X\to \Rp$ is $\M$-rearrangeable if $\M(\{f>t\})<\infty$ for every $t>0$. The $\m$-Scwhartz rearrangement of a rearrangeable function $f:X\to\Rp$ is the (unique) function $f^{*,\m}:\Rn\to\Rp$ such that
    \begin{enumerate}[label = (\Roman*)]
        \item\label{radial} $f^{*,\m}(\xi)=\rho(|\xi|)$ for some non-increasing function $\rho:\Rp\to\Rp$;
        \item $\m(\{f^{*,\m}>t\})=\M(\{f>t\})$ for every $t>0$.
    \end{enumerate}
    If $E\subseteq X$ is measurable and $\M(E)<\infty$, then $E^{*,\m}\subseteq\Rn$ is the subset such that $\chi_{E^{*,\m}}= (\chi_E)^{*,\m}$.
\end{definition}

Existence of the $\m$-Schwartz symmetrization follows from $\m(\Rn)=\infty$ (or $\m(\Rp)=\infty$, if $N=1$) and absolute continuity of $\m$ with respect to the Lebesgue measure; uniqueness is a consequence of $\frac{d\m}{d\scrL^n}>0$ $\scrL^N$-a.e. in $\Rn$.

Observe that property \ref{radial} in the above definition implies that the upper level sets $\{f^{*,\m}>t\}\subseteq\Rn$ are Euclidean balls centered in the origin for every $t>0$.

When the choice of the integer $N$ is clear and the measure $\m=\scrL^{N}$ is the Lebesgue measure in $\Rn$, we write $f^*$ and $E^*$ instead of $f^{*,\scrL^N}$ and $E^{*,\scrL^{N}}$ respectively.

The following proposition summarizes the properties of the Schwartz rearrangement required for the results presented in this paper. Kesavan  in \cite{kesavan2006symmetrization} provides proofs of these properties in the Euclidean setting; these can be straightforwardly adapted to our context.

\begin{proposition}\label{prop:prop}
    Let $(X,\M)$ be a measure space and fix $\m\in \calM(\R^N)$. Consider a rearrangeable function $f:X\to \Rp$ and let $f^{*,\m}:\Rn\to\Rp$ be its $\m$-Schwartz rearrangement. Then:
    \begin{enumerate}[label = (\roman*)]
        \item \label{prop:norms} For every $1\leq p < \infty$ the equality $\norm{f}_{L^p(X,\M)}=\norm{f^{*,\m}}_{L^p(\Rn,\m)}$ holds.
        \item \label{prop:notproved} The $\m$-Schwartz rearrangement induces a continuous operator $*:L^p(X,\M)\to L^p(\Rn,\m)$.
    \end{enumerate}
\end{proposition}

\subsection{Properties of the rearrangement on submanifolds}\label{ch:Submanifolds}

If $\Sigma$ is a smooth $n$-dimensional submanifold with boundary of $\Rnm$, then $\Sigma$ has a natural structure of measure space $(\Sigma, \Hn\llcorner \Sigma)$, where $\Hn\llcorner \Sigma$ is the restriction of the Hausdorff $n$-dimensional measure $\scrH^n$ to $\Sigma$. Unless explicitly stated otherwise, we assume throughout this work that the measure on  $\Sigma$ is $\Hn\llcorner \Sigma$. As in the Euclidean case, we abbreviate $\int_\Omega u\,d{\mathscr{H}^n\llcorner{\Sigma}}$ with $\int_\Omega u$ for every integrable function $u$ and measurable subset $\Omega\subseteq \Sigma$. Notice that when $\Sigma$ has smooth boundary $\partial \Sigma$, then $\partial \Sigma$ is a $(n-1)$-dimensional submanifold itself, hence we always endow it with the measure $\scrH^{n-1}\llcorner\partial \Sigma$ and use the same notation for integration.

Whenever $\Sigma$ is a smooth $n$-dimensional submanifold with boundary, its smooth structure allows us to consider the gradients of smooth function $u$ defined in $\Sigma$, namely 
\begin{align}\label{def:gradient}
    \nabla_\Sigma u(x):=\sum_{j=1}^n\partial_{w_j}u(x)w_j\in T_x\Sigma,
\end{align}
where $\{w_1,...,w_n\}$ is any orthonormal basis of the tangent space $T_x\Sigma$. We denote by $H:\Sigma \to \Rnm$ the mean curvature of $\Sigma$.

Fix $1\leq p <\infty$. We say that a function $u\in L^p(\Sigma)$ is \emph{$p$-Sobolev} and write $u\in W^{1,p}(\Sigma)$ if there exists a vector valued function $\nabla_\Sigma u\in L^p(\Sigma)$, called \emph{weak gradient of $u$} such that
    \begin{align*}
        \int_\Sigma u\,\div_\Sigma X = -\int_\Sigma X\cdot\nabla_\Sigma u\quad \forall X\in \mathfrak{X}_0(\Sigma),
    \end{align*}
    where $\mathfrak{X}_0(\Sigma)$ is the set of tangent vector  fields on $\Sigma$ that vanish on $\partial \Sigma$.
    We then define the \emph{$p$-Sobolev norm} $\norm{\cdot}_{W^{1,p}}$ in $W^{1,p}(\Sigma)$ to be 
    \begin{align*}
        \norm{u}_{W^{1,p}}:=\norm{u}_{L^p} + \norm{\nabla_\Sigma u}_{L^p}\quad \forall u\in W^{1,p}(\Sigma).
    \end{align*}
    Finally, the family $W^{1,p}_0(\Sigma)$ of \emph{$p$-Sobolev functions vanishing at $\partial \Sigma$} is defined as the topological closure of $C^\infty_c(\Sigma)\subseteq W^{1,p}(\Sigma)$ with respect to the norm $\norm{\cdot}_{W^{1,p}}$.

\begin{lemma}\label{lemma:approx}
    Consider a smooth $n$-dimensional submanifold with boundary $\Sigma\subseteq \Rnm$ and let $\Omega\subseteq \Sigma$ be an open subset with $\scrH^n(\Omega)<\infty$. Fix a real number $p\geq 1$. For any non-negative function $u\in\Lip_c(\Omega;\Rp)$ with $\int_\Sigma |\nabla_\Sigma u|^p<\infty$ there exists a sequence $(u_j)\subseteq\Lip_c(\Omega;\Rp)$ with $|\nabla_\Sigma u_j|>0$ $\scrH^n$-a.e. in $\{u_j>0\}$ and such that $u_j\to u$ in $W^{1,p}(\Sigma)$. In particular, 
    \begin{align*}
        \lim_{j\to\infty}\int_\Sigma |\nabla_\Sigma u_j|^p = \int_\Sigma |\nabla_\Sigma u|^p.
    \end{align*}
\end{lemma}

We refer \cite[Lemma 3.6]{Mondino_2020} for a proof of Lemma \ref{lemma:approx} in a more general setting. 

Recall the definition of $\calM(\R^N)$ given in the beginning of this section.

\begin{lemma}\label{lemma:techical}
    Consider the measure spaces $(\Sigma,{\mathscr{H}^n\llcorner{\Sigma}})$ and $(\R^N,\m)$, where $\Sigma\subseteq \Rnm$ is a smooth $n$-dimensional submanifold with boundary and $\m\ = f\scrL^N \in \calM(\R^N)$. Let $u\in W^{1,1}_0(\Sigma;\Rp)$ such that $|\nabla_\Sigma u|>0$ $\scrH^n$-a.e. in $\{u>0\}$. 
    \begin{enumerate}[label = (\roman*)]
        \item\label{lemma:technical_i} The function $t\mapsto \Hn(\{u>t\})$ is absolutely continuous, with derivative 
        \begin{align*}
            \frac{d}{dt}\Hn(\{u>t\}) = -\int_{\{u=t\}}\frac{1}{|\nabla_\Sigma u|}<0 \quad \ae\,\,t\in(0,\infty).
        \end{align*}
        \item If $u^{*,\m}:\R^N\to\Rp$ is the $\m$-Schwartz symmetrization of $u$ and  $\rho:\Rp\to\Rp$ such that $u^{*,\m}=\rho(|\cdot|)$, then $\rho$ is strictly decreasing and absolutely continuous.
    \end{enumerate}
\end{lemma}

\begin{proof}
    (i) As $|\nabla_\Sigma u|>0$ $\Hn$-a.e. in $\{u>0\}$, thanks to the co-area formula \cite{evans2018measure,lieb2001analysis} we can write
    \begin{align*}
        \Hn(\{u>t\}) = \int_{\{u>t\}}\frac{|\nabla_\Sigma u|}{|\nabla_\Sigma u|} = \int_t^\infty\left(\int_{\{u=s\}}\frac{1}{|\nabla_\Sigma u|}\right)\,ds.
    \end{align*}
    This proves that the function $t\mapsto \Hn(\{u>t\})$ is absolutely continuous. Hence, by the fundamental theorem of calculus for Lebesgue integrals, this function is differentiable at a.e. $t>0$ and its derivative at any such $t$ is $-\int_{\{u=t\}}1/|\nabla_\Sigma u|<0$. This ends the proof of the first claim.

    (ii) First, we prove that $\rho$ is strictly decreasing. By contradiction, suppose that $\rho\equiv C>0$ is constant on the interval $(a,b)$, for some $0<a<b$. Recalling that $\m=f \scrL^N$ for some $0<f\in L^1_\loc(\R^N)$, we deduce that $\m(\{u^{*,\m}=C\})>0$. On the other hand, for every $h>0$,
    \begin{align*}
        \m(\{u^{*,\m}=C\}) \leq \m(\{C-h<u^{*,\m}\leq C+h\}) = \Hn(\{C-h < u \leq C+h\}).
    \end{align*}
    Hence, passing to the limit as $h\to0$ and using (i) we obtain a contradiction.

    Assume $\Tilde{r}>0$ to be a discontinuity point for $\rho$. Set $\rho_\pm:=\lim_{r\to \Tilde{r}^\pm}\rho(r)$ with $\rho_+<\rho_-$. Then
    \begin{align*}
        0=\m(\{\rho_+ < u^{*,\m} < \rho_-\})=\Hn(\{\rho_+< u < \rho_-\}).
    \end{align*}
    This is contradicts the continuity of $u$. Therefore, $u^{*,\m}$ is continuous.

    We now prove the absolute continuity of $\rho$. Consider the function
    $\Psi:\Rp\to \Rp$,
    \begin{align*}
        \Psi(r):=\int_0^r \left(\int_{\partial B_s}f\right)\,ds.
    \end{align*}
    The function $\Psi$ is absolutely continuous, with a positive derivative a.e. in $(0,\infty)$, $\Psi(0)=0$ and $\lim_{r\to\infty}\Psi(r)=\infty$. This implies that $\Psi$ admits an absolutely continuous inverse $\Psi^{-1}:\Rp\to \Rp$.
    Let us also define $\tau:\Rp\to\Rp$ by the condition 
    \begin{align}\label{lemma:tau}
        \m(B_{\tau(t)})=\Hn(\{u>t\}).
    \end{align}
    Thanks to (i) and the assumption $\m\in \calM(\R^N)$, $\tau$ is a continuous and strictly decreasing function, hence $\tau$ is invertible. Moreover, as $\tau(t)=r$ if and only if $\rho(r)=t$, the inverse of $\tau$ is $\rho$. 
    Combining the co-area formula in $\R^N$, the symmetry of $f$ and the definition of $\Psi$, we can write
    \begin{align*}
        \m(B_{\tau(t)}) = \int_0^{\tau(t)}\left(\int_{\partial B_s} f \right)\,ds = \Psi(\tau(t)).
    \end{align*}
    Therefore, recalling the relation \eqref{lemma:tau},
    \begin{align*}
        \tau(t)= \Psi^{-1}(\Hn(\{u>t\})).
    \end{align*}
    This last expression, together with (i), establishes absolute continuity for $\tau = \rho^{-1}$. This suffices to end the proof. 

\end{proof}

}

\section{Proof of Theorem \ref{th:PS}}\label{sec:ProofThPs}

For every $n\geq 2$, the isoperimetric constants $\ic(n)$ is defined as
\begin{align}\label{def:isopConst}
    \ic(n):=\inf\left\{C : C\textn{ depends only on } n \textn{ and }\sup_{m,\Sigma}\frac{(\Hn(\Sigma))^{\frac{n-1}{n}}}{\left(\Hp(\partial \Sigma) + \int_\Sigma |H|\right)}\leq C\right\},
\end{align}
where the supremum is taken among all $m\geq 0$ and $n$-dimensional submanifolds of $\Rnm$.
Therefore, $\ic(n)\leq 5^n/\omega_n^{1/n}$ (cf. \cite{michael1973sobolev} or \cite[Theorem 5.7 of Chapter 4]{simon1984lectures}) for every integer $n\geq 2$ and
\begin{align}\label{nonSharp}
    (\Hn(\Omega))^{\frac{n-1}{n}}\leq \ic(n) \left(\Hp(\partial \Omega) + \int_\Omega |H|\right)
\end{align}
for every $n$-dimensional $\Sigma\subseteq\Rnm$ and measurable $\Omega\subseteq \Sigma$. The equivalent, functional version writes
\begin{align}
    \left(\int_\Sigma u^{\frac{n}{n-1}}\right)^\frac{n-1}{n}\leq \ic(n)\int_\Sigma(|\nabla_\Sigma u| + u |H|)
\end{align}
for every function $u\in W^{1,1}_0(\Sigma;\Rp)$.

Recall that $\Sigma$ belongs to $\stc{K}$ if the \emph{total mean curvature} of $\Sigma$, namely
    \begin{align*}
        \tc(\Sigma):= \norm{H}_{L^n(\Sigma)}
    \end{align*}
is not greater than $K$.

If $K_1\leq K_2$, then $\stc{K_1}\subseteq\stc{K_2}$. The class $\stc{K}$ is clearly invariant under translations and isometries of $\Rnm$; moreover, using some elementary properties of the mean curvature and the change of coordinates for submanifolds, one readily proves that $\stc{K}$ it is also invariant under rescaling.

For the rest of this section, $\Sigma$ will always be a $n$-dimensional submanifold in $\Rnm$ with smooth boundary.

\begin{proposition}[Isoperimetric/Sobolev inequality for $\Sigma$ in $\stc{K}$]\label{prop:isoStc}
    If $\Sigma \in \stc{K}$ for some $K<1/\ic(n)$, then, for every $n$-dimensional submanifold with boundary $\Omega \subseteq \Sigma$, the following inequality holds
    \begin{align}\label{Claim:isoStc}
        \left(\Hn(\Omega)\right)^{\frac{n-1}{n}} \leq \I(n,K) \Hp(\partial \Omega),\quad \I(n,K):=\frac{\ic(n)}{1-\ic(n)\,K}.
    \end{align}
    In particular, $\Hp(\partial \Sigma)>0$.

As a consequence, the functional counterpart of \eqref{Claim:isoStc}, that is
    \begin{align*}
        \norm{u}_{L^{\frac{n}{n-1}}}\leq \I(n,K) \int_\Sigma |\nabla_\Sigma u|\quad \forall u\in W^{1,1}_0(\Sigma),
    \end{align*}
    holds.
\end{proposition}

\begin{proof}
    Recall \eqref{nonSharp} and use H\"older's inequality together with the definition of total mean curvature to obtain
    \begin{align*}
        \left(\Hn(\Omega)\right)^{\frac{n-1}{n}}\left(1-\ic(n)\tc(\Omega)\right) \leq \ic(n) \Hp(\partial \Omega).
    \end{align*}
    Therefore, as $\tc(\Omega)\leq K < 1/\ic(n)$, \eqref{Claim:isoStc} holds. 
    
    The functional counterpart is recovered by using the standard machinery of Lebesgue integration.
    
\end{proof}

We are finally ready to prove the main theorem of this paper.

\begin{proof}[Proof of Theorem \ref{th:PS}]
    We begin by showing the result for $p=1$. Without loss of generality, we can assume $u\in \Lip_c(\Sigma)$ and $|\nabla_\Sigma u|>0$ in $\{u>0\}$. The general case will then follow by an approximation argument involving the continuity of the Schwartz rearrangement (Proposition \ref{prop:prop} \ref{prop:notproved}) and Lemma \ref{lemma:approx}. Also note that in this case, by virtue of Lemma \ref{lemma:techical}, $u^*$ is differentiable $\Ln$-a.e. .
    
    Define the functions $\phi,\phi_*:\Rp\to\Rp$ as
    \begin{align*}
        \phi(t):=\int_{\{u>t\}}|\nabla_\Sigma u|,\quad \phi_*(t):=\int_{\{u^*>t\}} |\nabla u^*|.
    \end{align*}
    Using the co-area formula as in the proof of Lemma \ref{lemma:techical} \ref{lemma:technical_i}, we deduce 
    \begin{align}\label{proofPS:1}
        -\phi'(t)= \Hp(\{u=t\}),\quad -\phi_*'(t)=\Hp(\{u^*=t\})\quad \ae\,t>0.
    \end{align}
    On the one hand, by virtue of Proposition \ref{prop:isoStc} applied to the sets $\Omega_t:=\{u>t\}$ for a.e. $t>0$, 
    \begin{align}\label{proofPS:2}
        \Hp(\{u=t\}) \geq \frac{1}{\I(n,K)}\left(\Hn(\{u>t\})\right)^\frac{n-1}{n}.
    \end{align}
    On the other hand, as $u^*$ is radial, the equality case in the Euclidean isoperimetric inequality applies to every level set of $u^*$. Hence,
    \begin{align}\label{proofPS:3}
        \Hp(\{u^*=t\}) = n\omega_n^{1/n}\left(\Ln(\{u^*>t\})\right)^{\frac{n-1}{n}}.
    \end{align}
    By definition  the $\Hn$-measure of each upper level set of $u$ coincide with the $\Ln$-measure of the corresponding level set of $u^*$, \eqref{proofPS:1}, \eqref{proofPS:2} and \eqref{proofPS:3} give
    \begin{align}\label{th:+}
        -\phi_*'(t) \leq \I(n,K)n\omega_n^{1/n}(-\phi'(t))\quad \ae\,t>0.
    \end{align}
    Integrating this inequality from $t=0$ to $t=\infty$ and recalling that both $\phi$ and $\phi_*$ are decreasing and vanishing at $\infty$, yields
    \begin{align*}
        \int_{\Rn}|\nabla u^*|\leq \ps(n,K)\int_\Sigma |\nabla_\Sigma u|,
    \end{align*}
    with the definition $\ps(n,K):=\I(n,K)n\omega_n^{1/n}$.

    Let $\mu,\mu_*:\Rp\to\Rp$ be the functions
    \begin{align*}
        \mu(t):=\Hn(\{u>t\}),\quad \mu_*(t):=\Ln(\{u^*>t\}).
    \end{align*}
    Recalling the definition of Schwartz rearrangement, we deduce that $\mu \equiv \mu_*$. Also, thanks to the co-area formula,  Lemma \ref{lemma:techical} (ii) and arguing an in the proof of Lemma \ref{lemma:techical} (i), we get
    \begin{align}\label{th:triang}
        -\mu'(t)= \int_{\{u=t\}} \frac{1}{|\nabla_\Sigma u|},\quad-\mu_*'(t)=\int_{\{u^*=t\}} \frac{1}{|\nabla u^*|} \quad \ae\,t>0.
    \end{align}
    Let us write, in accordance with Lemma \ref{lemma:techical} (ii), $u^* = \rho(|\cdot|)$ with $\rho:\Rp\to\Rp$ strictly decreasing and absolutely continuous. In particular, as the level sets $\{u^*=\lambda\}$ are $(n-1)$-spheres in $\Rn$ for every $\lambda >0$,
    \begin{align}\label{th:nablau*}
        |\nabla u^*| = -\rho'(|\cdot|)\quad \text{and}\quad -\mu_*'(t)= \frac{\Hp(\{u^*=t\})}{-\rho'(\rho^{-1}(t))}\quad \ae\,t>0.
    \end{align}
    Fix $t>0$ such that all of the derivative $\phi',\phi_*',\mu',\mu_*'$ exist at $t$ (these $t$'s form a full measure subset of $\Rp$) and fix $\lambda\in \Rp$. Then, by H\"{o}lder's inequality,
    \begin{align*}
        \frac{\phi(t)-\phi(t+\lambda)}{\lambda}\leq \Big(\frac{1}{\lambda}\int_{\{t<u<t+\lambda\}}|\nabla_\Sigma u|^p\Big)^\frac{1}{p}\Big(\frac{\mu(t)-\mu(t+\lambda)}{\lambda}\Big)^{\frac{1}{p'}}.
    \end{align*}
    Pass to the limit as $\lambda\to 0$ in the above inequality. Using the co-area formula in the first integral of the right-hand side,  \eqref{th:+} and some algebraic manipulations, we write
    \begin{align}\label{triang}
        \Big(\int_{\{u=t\}}|\nabla_\Sigma u|^{p-1}\Big)^{\frac{1}{p}} \geq \frac{-\phi'(t)}{(-\mu'(t))^{1/p'}}\geq \frac{1}{\ps(n,K)} \frac{-\phi_*'(t)}{(-\mu_*'(t))^{1/p'}}.
    \end{align}
    The last expression is justified by the fact that, as we are assuming $|\nabla_\Sigma u|>0$ in $\{u>0\}$, we have $-\mu'(t)>0$. Combining \eqref{proofPS:1} and \eqref{th:triang}, we obtain 
    \begin{align}\label{th:dotdot}
        \frac{-\phi_*'(t)}{(-\mu_*'(t))^{1/p'}} = \left(\left(-\rho'(\rho^{-1}(t))\right)^{p-1}\Hp(\{u^*=t\})\right)^{\frac{1}{p}}   = \left(\int_{\{u^*=t\}}|\nabla u^*|^{p-1}\right)^\frac{1}{p}
    \end{align}
    (note that the assumption $|\nabla_\Sigma u|>0$ $\Hn$-a.e. in $\{u>0\}$ is being used to justify the existence of $\rho^{-1}$ in the interval $(0,\infty)$). 
    Recalling \eqref{triang}, we reach
    \begin{align*}
        \int_{\{u^*=t\}}|\nabla u^*|^{p-1}\leq (\ps(n,K))^p \int_{\{u=t\}}|\nabla_\Sigma u|^{p-1}.
    \end{align*}
    Finally, we integrate this last inequality from $t=0$ to $t=\infty$ and use the co-area formula  to obtain
    \begin{align*}
        \int_{\Rn}|\nabla u^*|^{p}\leq (\ps(n,K))^p\int_{\Sigma}|\nabla_\Sigma u|^{p},
    \end{align*}
    which is exactly \eqref{claim:PS}.
    
\end{proof}

\begin{remark}
\newcommand{\p}{\mathfrak{p}}
    Note that $\ps(n,K)=(\ic(n)n\omega_n^{1/n})/(1-\ic(n)\,K)$ in \eqref{claim:PS} \emph{does not} depend on $p$. 

\end{remark}

\section{Consequences of Theorem \ref{th:PS} and proof of Corollary \ref{th:SLS}}\label{sec:Cor}

This method for translating the gradient of a function defined on a manifold into the gradient of another function in Euclidean space entails several significant consequences.  Let us start with the proof of the monotonicity principle stated in the introduction.

\begin{corollary}[Integral inequalities of the first order]\label{cor:MonPrinc}
    Let $\Sigma\in \stc{K}$ for some $K<1/\ic(n)$ and consider the functions $f,g,\phi,\psi:\Rp\to\R$ and $L,\Lambda:\Rp\times\Rp\to\R$, where
    \begin{enumerate}
        \item $t\mapsto L(s,t)$ is non-increasing for any $s\in \Rp$;
        \item $t\mapsto \Lambda(s,t)$ is non-decreasing for any $s\in \Rp$;
        \item $f$ and $\phi$ are continuous, strictly increasing and $f(0)=g(0)=0$;
        \item $g$ and $\psi$ are of the form 
        \begin{align*}
            g(t)=\sum_{j=P_1}^{Q_1} b_j t^{p_j}\quad\textn{and}\quad\psi(t)=\sum_{j=P_2}^{Q_2} c_j t^{q_j}
        \end{align*}
        where $1\leq P_i\in \mathbb{N}$, $P_i<Q_i\in \mathbb{N}\cup\{\infty\}$ for $i=1,2$, $b_j\leq 0\leq c_j$,  $1\leq p_j < p_{j+1}$ for every $P_1\leq j < Q_1+1$ and $1\leq q_j < q_{j+1}$ for every $P_2\leq j < Q_2+1$,  with the convention $\infty + 1 := \infty$;
    \end{enumerate}
    Suppose that a radially symmetric non-increasing function $v\in \Lip_c(\Rn)$ satisfies 
    \begin{align}\label{monPrinc:Ass1}
        L\left(\int_\Rn f(v),\int_\Rn g(|\nabla v|)\right) \leq \Lambda\left(\int_\Rn \phi(v),\,\int_\Rn \psi(|\nabla v|)\right).
    \end{align}
    Then, every function $u\in \Lip_c(\Sigma)$ such that $u^*=v$ satisfies
    \begin{align}\label{monPrinc:Claim}
        L\left(\int_\Sigma f(u),\int_\Sigma g(\ps(n,K)\nabla_\Sigma u)\right) \leq \Lambda\left(\int_\Sigma \phi(u),\,\int_\Sigma \psi(\ps(n,K) |\nabla_\Sigma u|)\right).
    \end{align}
\end{corollary}

\begin{proof}
    First, notice that, if $F:\Rp\to\Rp$ is continuous, strictly increasing and vanishes at zero, we have
    \begin{align}\label{monPrinc:1}
        (F(u))^* = F(u^*) \quad \forall u\in C^\infty_c(\Sigma;\Rp).
    \end{align} 
    To prove this, fix $u\in \Lip_c(\Sigma)$. The function $(F(u))^*$ is characterized by being the unique radially symmetric non-increasing function such that the $\Ln$-measure of each of its upper-level sets of equals the $\Hn$-measure of the corresponding upper-level sets of $F(u)$. As $u^*$ is a radially symmetric non-increasing function and $F$ is non-decreasing, $F(u^*)$ is radially symmetric non-decreasing too; moreover, since $F$ vanishes at $0$ and is strictly increasing, it admits a continuous and strictly increasing inverse $F^{-1}:\Rp\to\Rp$. Fix $t>0$. Then
    \begin{align*}
    \begin{split}
        \Ln(\{F(u^*)>t\}) = \Ln(\{u^*> F^{-1}(t)\})
        = \Hn(\{u> F^{-1}(t)\})
        = \Hn(\{F(u)>t\}),
    \end{split}
    \end{align*}
    which proves \eqref{monPrinc:1}.

    Suppose now that $u^*$ satisfies \eqref{monPrinc:Ass1}. Recalling Proposition \ref{prop:prop} \eqref{prop:norms}, standard theorems for the exchange of limits, \eqref{monPrinc:Ass1} and Theorem \ref{th:PS}, we obtain
    \begin{align*}
    \begin{split}
        L&\left(\int_\Sigma f(u),\int_\Sigma g(\ps(n,K)|\nabla_\Sigma u|)\right) = L\left(\int_\Sigma f(u),\sum_{j=P_1}^{Q_1} b_j \ps(n,K)^{p_j}\int_\Sigma |\nabla_\Sigma u|^{p_j}\right)\\
        &\qquad\qquad\qquad\qquad\qquad\leq L\left(\int_\Rn f(u^*),\sum_{j=P_1}^{Q_1} b_j \int_\Rn |\nabla u^*|^{p_j}\right)\\
        &\qquad\qquad\qquad\qquad\qquad\leq \Lambda\left( \int_\Rn \phi(u^*),\, \sum_{j=P_2}^{Q_2} c_j \int_\Rn |\nabla u^*|^{q_j}\right) \\
        &\qquad\qquad\qquad\qquad\qquad\leq \Lambda\left( \int_\Rn \phi(u^*),\,\sum_{j=P_2}^{Q_2} c_j \ps(n,K)^{q_j} \int_\Sigma |\nabla_\Sigma u|^{q_j}\right)\\
        &\qquad\qquad\qquad\qquad\qquad= \Lambda\left(\int_\Rn \phi(u^*),\,\int_\Sigma \psi(\ps(n,K) |\nabla_\Sigma u|)\right).
    \end{split}
    \end{align*}
This ends the proof.

\end{proof}

As Talenti in \cite{talenti1976best} proved that for every function $v\in W^{1,p}(\Rn)$
\begin{align}\label{cor:pSob_1}
        \norm{v}_{L^{p^*}}\leq \at(n,p)\norm{\nabla v}_{L^p},
\end{align}
where $\at(n,p)$ is the Talenti's sharp constant
\begin{align*}
        \at(n,p):= \frac{1}{\sqrt{\pi}n^{1/p}}\left(\frac{p-1}{n-p}\right)^{1-\frac{1}{p}}\left(\frac{\Gamma(1+n/2)\Gamma(n)}{\Gamma(n/p)\Gamma(1+n-n/p)}\right)^\frac{1}{n}.
\end{align*}
Using the monotonicity principle (or simply Theorem \ref{th:PS}), one can readily prove the following corollary.

\begin{corollary}[$p$-Sobolev inequality for $\Sigma$ in $\stc{K}$]\label{cor:pSobStc}
    Let $1< p<n$ and let $p^*:=np/(n-p)$ be its Sobolev conjugate. If $\Sigma\in \stc{K}$ with $K<1/\ic(n)$ and $u\in W^{1,p}_0(\Sigma)$, then
    \begin{align*}
        \norm{u}_{L^{p^*}}\leq \S(n,p,K)\norm{\nabla_\Sigma u}_{L^p}, \quad \S(n,p,K)=\at(n,p)\ps(n,K).
    \end{align*}
\end{corollary}

Cordero-Erasquin, Nazaret and Villani in \cite{cordero2004mass}  and by Del Pino and Dolbeault in \cite{del2003optimal} proved a sharp Euclidean Gagliardo-Nirenberg inequality. The statement of this is included below.

\begin{lemma}[Sharp Euclidean Gagliardo-Nirenberg inequality]\label{lemma:SEGN}
    Fix $1<p<n$, $p<q\leq \frac{p(n-1)}{n-p}$ and set $r:=\frac{p(q-1)}{p-1}$. For every function $v\in W^{1,p}(\Rn)$, 
    \begin{align}\label{lemma:sharpEGNI}
       \norm{v}_{L^r}\leq \egn(n,p,q)\norm{\nabla v}_{L^p}^\theta\norm{v}_{L^q}^{1-\theta}
    \end{align}
    holds, where, setting $\beta:=\beta(n,p,q):= np-q(n-p)>0$, the constants $\egn$ and $\theta$ are
    \begin{gather}
        \theta=\theta(n,p,q):=\frac{n(q-p)}{(q-1)\beta}\label{theta}\\
        \egn(n,p,q):= 
        \left(\frac{q-p}{p\sqrt{\pi}}\right)^\theta
        \left(\frac{pq}{n(q-p)}\right)^{1-\theta}
        \left(\frac{\beta}{pq}\right)^\frac{1}{r}
        \left(\frac{\Gamma\left(q\frac{p-1}{1-p}\right)\Gamma\left(\frac{n}{2}+1\right)}{\Gamma\left(\frac{(p-1)\beta}{p(q-p)}\right)\Gamma\left(n\frac{p-1}{p}+1\right)}\right)^\frac{\theta}{n}.\label{egn}
    \end{gather}
    The inequality \eqref{lemma:sharpEGNI} is optimal and equality holds if and only if there exist $a\in\R,b>0$ and $\overline{x}\in \Rn$ such that 
    \begin{align*}
        v(x)= a\left(1+b|x-\overline{x}|^\frac{p}{p-1}\right)^{-\frac{p-1}{q-p}}\quad\forall x\in \Rn.
    \end{align*}
\end{lemma}

Combining Corollary \ref{cor:MonPrinc} with Lemma \ref{lemma:SEGN}, we immediately obtain a Gagliardo-Nirenberg inequality for submanifolds with total mean curvature smaller than $1/\ic(n)$. More precisely, we obtain the following corollary.

\begin{corollary}[Gagliardo-Nirenberg for $\Sigma$ in ${\stc{K}}$]\label{cor:GN_Stc}
    Fix $1<p<n$, $p<q\leq \frac{p(n-1)}{n-p}$ and set $r:=\frac{p(q-1)}{p-1}$. If $\Sigma\in \stc{K}$ for some $K<1/\ic(n)$ and $u\in W^{1,p}_0(\Sigma)$ then 
    \begin{align}\label{cor:GNStc}
        \norm{u}_{L^r}\leq \gn(n,p,q,K)\norm{\nabla_\Sigma u}_{L^p}^\theta\norm{u}_{L^q}^{1-\theta}
    \end{align}
    where $\theta$ is defined in \eqref{theta} and
    \begin{align*}
        \gn(n,p,q,K) := \egn(n,p,q)\ps(n,K).
    \end{align*}
\end{corollary}

In a similar manner, using the spectral gap inequality in the Euclidean setting (see \cite{faber1923beweis}) together with our monotonicity principle, we recover the following result.

\begin{corollary}[Spectral Gap for $\Sigma$ in $\stc{K}$]\label{cor:SGStc}
    Let $\Sigma\in \stc{K}$ with $K<1/\ic(n)$. Fix a  bounded open subset $\Omega\subseteq\Sigma$ and $u\in W^{1,2}_0(\Omega)$ different from the zero function. Then 
    \begin{align*}
        \frac{\int_{\Omega}|\nabla_\Sigma u|^2}{\int_\Omega u^2} \geq \frac{\g(n,K)}{\Hn(\Omega)^{2/n}},\quad \g(n,K):=\frac{\mathfrak{j}_{\frac{n}{2}-1}}{\ps(n,K)}\omega_n^{2/n}, 
    \end{align*}
    where $\mathfrak{j}_{k}$ is the first positive zero of the Bassel function $\mathfrak{J}_k$. In particular, 
    \begin{align*}
        \inf\left\{\frac{\int_{\Omega}|\nabla u|^2}{\int_\Omega u^2}:u\in W^{1,2}_0(\Omega), u\not \equiv 0\right\}\to\infty \textn{ as } \Hn(\Omega)\to0
    \end{align*}
\end{corollary}

For applications, it may be useful to have a quantitative estimates for the the constants $\I(n,K)$, $\ps(n,K)$, etc. at the cost of introducing the dependence on the codimension $m$.

 In the specific cases of codimension $m=1$ and $m=2$, Brendle \cite{brendle2021isoperimetric} (see also \cite{BrendleEichmair+2023+1+10,brendle2024isoperimetric}) proved that the sharp isoperimetric constant is the Euclidean one, namely $1/(n\omega_n^{1/n})$. While determining the sharp constant remains an open problem for higher codimensions ($m \geq 3$), Brendle's work provides the following isoperimetric inequality, where the constant depends on the codimension:
\begin{gather*}
    (\Hn(\Omega))^{\frac{n-1}{n}}\leq \be(n,m) \left(\Hp(\partial \Omega) + \int_\Omega |H|\right), \\\quad \be(n,m):=
    \begin{cases}
    \frac{1}{n\omega_n^{1/n}}&,\textn{ if } m\in \{1,2\}\\
    \min\left\{\frac{1}{n}\left(\frac{m\omega_m}{(n+m)\omega_{n+m}}\right)^{\frac{1}{n}},\frac{5^n}{\omega_n^{1/n}}\right\}&,\textn{ if }m\geq 3
    \end{cases}
\end{gather*}
for every measurable subset $\Omega \subseteq \Sigma$. Using $\be(n,m)$ instead of $\ic(n)$ in the previous  results give, for any submanifold with total mean curvature bounded by $K<1/\be(n,m)$, the constants
\begin{gather*}
    \Tilde{\I}(n,m,K)= \frac{\be(n,m)}{1-\be(n,m)K},\quad \Tilde{\ps}(n,m,K)=\Tilde{\I}(n,m)n\omega_n^{1/n},
\end{gather*}
and so on.

\begin{proposition}\label{prop:asymtoticSharpness}
    Let $\Sigma\in\stc{K}$ with $K< n\omega_n^{1/n}$ . If $m\in\{1,2\}$ then 
    \begin{align}\label{prop:sharp}
        \norm{\nabla u^*}_{L^p(\Rn)}\leq \frac{n\omega_n^{1/n}}{n\omega_n^{1/n}-K} \norm{\nabla_\Sigma u}_{L^p(\Sigma)}
    \end{align}
    holds for every $u\in W^{1,p}_0(\Sigma;\Rp)$ and every $1\leq p < \infty$. Moreover,
    \begin{enumerate}[label = (\roman*)]
        \item if $\Sigma$ is a minimal submanifolds, then \eqref{prop:sharp} is sharp for every $n\geq 2$;
        \item in the general case, \eqref{prop:sharp} is asymptotically sharp as $n\to\infty$, in the sense that 
    \begin{align}\label{prop:sharpLimit}
        \lim_{n\to\infty} \frac{n\omega_n^{1/n}}{n\omega_n^{1/n}-K} =1.
    \end{align}
    \end{enumerate}
\end{proposition}

\begin{proof} 
\newcommand{\p}{\mathfrak{P}}
    When $m\in\{1,2\}$, inequality \eqref{prop:sharp} is derived by replacing $\ic(n)$ with $\be(n,m)$ in Theorem \ref{th:PS}. 
    
    Suppose $\Sigma$ is a minimal submanifold. Then $\Sigma \in \stc{0}$, hence \eqref{prop:sharp} is sharp. Assume now only $\Sigma \in \stc{K}$ for some $0<K<n\omega_n^{1/n}$. Using Stirling's approximation (see \cite{marsaglia1990new}), we write
    \begin{align*}
        n\omega_n^\frac{1}{n} = n^{\frac{1}{2}-\frac{1}{2n}}\,\pi^{-\frac{1}{2n}}\,\sqrt{2\pi e} + o_n(1) \quad \textn{as } n\to \infty.
    \end{align*}
    Since
    \begin{align*}
        \lim_{n\to\infty} n^{\frac{1}{2}-\frac{1}{2n}}\,\pi^{-\frac{1}{2n}}= \infty,
    \end{align*}
    then there exists $\overline{n}\in \mathbb{N}$ such that $n\omega_n^{1/n}>K$ for every $n\geq \overline{n}$. Thus, we can compute
    \begin{align*}
    \lim_{n\to\infty} \frac{n\omega_n^{1/n}}{n\omega_n^{1/n}-K} = \lim_{n\to\infty} \frac{n^{\frac{1}{2}-\frac{1}{2n}}\,\pi^{-\frac{1}{2n}}\,\sqrt{2\pi e}}{n^{\frac{1}{2}-\frac{1}{2n}}\,\pi^{-\frac{1}{2n}}\,\sqrt{2\pi e} - K} =1.
    \end{align*}
\end{proof}

\begin{remark}
(I) Arguing as in the proof of Proposition \ref{prop:asymtoticSharpness}, one shows that also the inequalities 
\begin{gather}
    \norm{u}_{L^{p^*}}\leq \at(n,p)\frac{n\omega_n^{1/n}}{n\omega_n^{1/n}-K}\norm{\nabla_\Sigma u}_{L^p},\notag \\
    \norm{u}_{L^r}\leq \egn(n,p,q)\frac{n\omega_n^{1/n}}{n\omega_n^{1/n}-K} \norm{\nabla_\Sigma u}_{L^p}^\theta \norm{u}_{L^q}^{1-\theta},\label{eq:egnWithB}\\
    \frac{\int_\Omega |\nabla_\Sigma u|^2}{\int_\Omega u^2}\geq \frac{n\omega_n^{1/n}-K}{n}\mathfrak{j}_{\frac{n}{2}-1}\omega_n^{1/n}\notag
\end{gather}
(as in Corollary \ref{cor:pSobStc}, Corollary \ref{cor:GN_Stc} and Corollary \ref{cor:SGStc} respectively) hold for $m\in \{1,2\}$ and are asymptotically sharp as $n\to\infty$.

(II) If $n\omega_n^{1/n}$ is the sharp constant for the isoperimetric inequality \eqref{IsoIneq} in any codimension, then Proposition \ref{prop:asymtoticSharpness} and all the inequalities above hold for any $m\geq 1$ and are asymptotically sharp as $n\to\infty$.
\end{remark}

\begin{proof}[Proof of Corollary \ref{th:SLS}]
    Let $\Sigma$ be a minimal submanifold of $\R^{n+m}$ with $m\in\{1,2\}$.
By virtue of Theorem \ref{th:PS} and relation \eqref{eq:egnWithB} in the remarks above, the proof of Corollary \ref{th:SLS} reduces to the proof of the sharp Euclidean log-Sobolev inequality provided by Del Pino and Dolbeault in \cite{del2003optimal}.

\end{proof}

\section{Asymptotic sharpness of the total mean curvature bound in Theorem \ref{th:PS}}\label{ch:Sharpness}

Observe that, for the case $p=1$, the following statement holds without any assumption on the total mean curvature of $\Sigma$.

\begin{proposition}
    Let $\Sigma\subseteq\Rnm$ be a $n$-dimensional submanifold. For every $u\in W^{1,1}_0(\Sigma;\Rp)$ let $u^*:\Rn\to\Rp$ be its Schwartz rearrangement.  Then
    \begin{align*}
        \int_\Rn |\nabla u^*| \leq \ic(n)n\omega_n^{1/n} \int_\Sigma \left(|\nabla_\Sigma u| + u |H|\right).
    \end{align*}
    In particular, if $u|H|\in L^1(\Sigma)$ then $u^*\in W^{1,p}(\Rn;\Rp)$.
\end{proposition}

\begin{proof}
    Repeat the argument in the first part of the proof of Theorem \ref{th:PS} with the functions $\phi,\phi_*:\Rp\to\Rp$
    \begin{align*}
        \phi(t):=\int_{\{u>t\}}\left(|\nabla_\Sigma u| + |H|(u-t)\right),\quad \phi_*(t):=\int_{\{u^*>t\}}|\nabla u^*|,
    \end{align*}
    and the Michael-Simon's inequality \eqref{IsoIneq} instead of Proposition \ref{prop:isoStc}.
    
\end{proof}

Our initial purpose was to establish an inequality of the type
\begin{align}\label{false}
    \int_{\Rn} |\nabla u^*|^p  \leq C(n,p) \int_\Sigma \left(|\nabla_\Sigma u|^p + |H|^pu^p\right),
\end{align}
for every $u\in W^{1,p}_0(\Sigma;\Rp)$. However, \eqref{false} cannot hold in general due to the following example.

\begin{example}\label{ex:counter}
    Let $\Sigma:={\mathbb{S}^2}$ be the 2-dimensional unit sphere in $\R^3$. With this choice, it is easy to show that $|H|\equiv 2$ and clearly $\partial {\mathbb{S}^2} = \varnothing$. For $\lambda \in [1,\infty)$, consider the function $\Tilde{u}_\lambda:\R^3\to\Rp$ given by
    \begin{align*}
        \Tilde{u}_\lambda(x,y,z):=
        \begin{cases}
            \lambda r&,\textn{ if } 0\leq r\leq \frac{1}{\lambda},\,z>0\\
            1&,\textn{ otherwise}
        \end{cases},
    \end{align*}
    where $r:=\sqrt{x^2+y^2}$. Define $u_\lambda:{\mathbb{S}^2}\to\Rp$  as the restriction of $\Tilde{u}_{\lambda}$ to ${\mathbb{S}^2}$. As the tangent space of ${\mathbb{S}^2}$ at the point $(0,0,1)$ is the plane $\R^2\times\{0\}$, then, for $\lambda$ large enough, we can write
    \begin{align}\label{ex:nablau}
        |\nabla_{\mathbb{S}^2} u_\lambda(x,y,z)| =
        \begin{cases}
            \lambda + o_\lambda(1)&,\textn{ if }0\leq r\leq \frac{1}{\lambda},\,z>0\\
            0&,\textn{ otherwise}
        \end{cases}
        \quad (x,y,z)\in {\mathbb{S}^2}.
    \end{align}
    Thus, for $p\geq 1$,
    \begin{align}\label{ex:intnablau}
        \begin{split}
            \int_{\mathbb{S}^2} |\nabla_{\mathbb{S}^2} u_\lambda|^p = \int_{B_{1/\lambda}} \frac{\left(\lambda + o(\lambda)\right)^p}{\sqrt{1-x^2-y^2}}\,d\scrL^2(x,y) = \pi \lambda^{p-2} + o(\lambda^{p-2}).
        \end{split}
    \end{align}

    After some elementary computations one shows that $u^*_\lambda(\cdot)=\rho_\lambda(|\cdot|)$, where
    \begin{align*}
        \rho_\lambda(s) = 
        \begin{cases}
            1&,\textn{ if } 0\leq s \leq \sqrt{2\left(1+\sqrt{1-\frac{1}{\lambda^2}}\right)}\\
            \lambda\sqrt{1-\left(\frac{s^2}{2}-1\right)^2}&, \textn{ if } \sqrt{2\left(1+\sqrt{1-\frac{1}{\lambda^2}}\right)} < s <2\\
            0 &,\textn{ if } 2\leq s
        \end{cases}.
    \end{align*}
    Therefore, letting $s:=(\xi^2+\zeta^2)^{1/2}$,
    \begin{align}\label{ex:nablaustar}
        |\nabla u^*_\lambda(\xi,\zeta)| = -\rho_\lambda'(s) = \begin{cases}
            0 &, \textn{ if } 0\leq s \leq \sqrt{2\left(1+\sqrt{1-\frac{1}{\lambda^2}}\right)}\\
            \lambda \frac{s(s^2-2)}{\sqrt{1-\left(\frac{s^2}{2}-1\right)^2}} &, \textn{ if } \sqrt{2\left(1+\sqrt{1-\frac{1}{\lambda^2}}\right)} < s < 2\\
            0 &, \textn{ if } 2 <s 
        \end{cases}.
    \end{align}
    Integrating \eqref{ex:nablaustar} using the spherical change of coordinates $(\xi,\zeta)\mapsto(s,\theta)$  and the substitution $y=(s^2/2-1)^2$, we obtain
    \begin{align}\label{ex:intnablaustar}
        \begin{split}
            \int_{\R^2}|\nabla u^*_\lambda|^p 
                =\begin{cases}
                        \frac{\pi 2^{2p-p/2} }{2-p} \lambda^{2p-2}+ o(\lambda^{2p-2}) &,\textn{ if }1\leq p <2\\
                        \infty  &,\textn{ if } 2\leq p.
                        \end{cases}
        \end{split}
    \end{align}

    Therefore, for $1\leq p <2$, \eqref{ex:intnablau} and \eqref{ex:intnablaustar} yield
    \begin{align*}
        \frac{\int_{\mathbb{S}^2} \left(|\nabla_{\mathbb{S}^2} u_\lambda|^p + |H|^pu_\lambda^p\right)}{\int_{\R^2}|\nabla u_\lambda^*|^p}\leq \kappa_p \frac{\lambda^{p-2} + 1}{\lambda^{2(p-1)}}\xrightarrow[\lambda\to\infty]{}\begin{cases}
            \kappa_1&,\textn{ if }p=1\\
            0&,\textn{ if }1<p<2
        \end{cases},
    \end{align*}
    for some constant $\kappa_p>0$ that depends only on $p$. While, for $p\geq 2$, $|\nabla u_\lambda^*|\not\in L^p(\R^2)$.

    This shows that \eqref{false} can not be true in general.
    
\end{example}

Observe that the construction in Example \ref{ex:counter} relies on the fact that $\Sigma$ is a submanifold without boundary. It can be seen that this calculation generalizes to any compact submanifold with empty boundary (note that the assumptions of Theorem \ref{th:PS}, thanks to Proposition \ref{prop:isoStc}, exclude these scenarios). Moreover, if $\mathbb{S}^n$ denotes the $n$-dimensional unit sphere in $\R^{n+1}$, $H\equiv n$. Therefore, 
\begin{align*}
    \tc(\mathbb{S}^n)= n(n\omega_n)^{1/n} = \frac{1}{\be(n,1)} + o_n(1) \textn{ as }n\to\infty.
\end{align*}
This proves that the assumption ``$\Sigma \in \stc{K}$ for some $K<1/\be(n,m)$'' is \emph{at least asymptotically sharp} when $m\in\{1,2\}$ as $n\to\infty$. Determining whether this constraint is actually sharp for each $n\geq 2$ in codimension $m=1$ and $m=2$ is still an open question, as a fortiori does the sharpness of ``$\Sigma \in \stc{K}$ for some $K<1/\ic(n)$''.

\section*{Appendix: another formulation of Theorem \ref{th:PS}}

\renewcommand{\u}{{u^{*,\m_{n,K}}}}
Another approach, analogous to the result of Mondino and Semola \cite{Mondino_2020} for $\cd(K,N)$ spaces, is the following. Given a $n$-dimensional submanifold $\Sigma$ with mean curvature $H$, following Milman \cite{milman2011sharp}, we consider a \emph{``model space''} $(\Rp,\m)$ with some sort of generalized notion of curvature and dimension, which makes it behave like a $n$-dimensional space with a constant mean curvature.

        Let $\m:\B(\Rp)\to[0,\infty]$ be any Radon measure on $\Rp$ and $1\leq p<\infty$. For every function $u\in L^p(\Rp,\m)$ we define the quantity
        \begin{align}\label{def:Sobm}
            \int |\nabla_\m u|^p\,d\m:=\inf\left\{\liminf_{j\to\infty} \int |u'_j|^p\,d\m : (u_j)_j\subseteq \Lip(\Rp)\cap L^p(\m),u_j\xrightarrow[]{L^p(\m)}u\right\}.
        \end{align}        
        Note that \eqref{def:Sobm} \emph{does not} give a function $\nabla_\m u:\Rp\to\Rp$. Instead, it only defines the symbol $\int |\nabla_\m u|^p\,d\m$. The space of $p$-Sobolev function in $(\Rp,\m)$ is the collection 
        \begin{align*}
            W^{1,p}(\Rp,\m):=\left\{u\in L^p(\Rp,\m):\int|\nabla_\m u|^p\,d\m <\infty\right\}.
        \end{align*}
        The $p$-norm in the space $W^{1,p}(\Rp,\m)$ is the function $\norm{\cdot}_{W^{1,p}(\Rp,\m)}:W^{1,p}(\Rp,\m)\to \Rp$ defined as
        \begin{align*}
            \norm{u}_{W^{1,p}(\Rp,\m)}:=\norm{u}_{L^p(\Rp,\m)} + \int |\nabla_\m u|^p\,d\m.
        \end{align*}
        One can prove that $\norm{\cdot}_{W^{1,p}(\Rp,\m)}$ is in fact a norm, hence $(W^{1,p}(\Rp,\m),\norm{\cdot}_{W^{1,p}(\Rp,\m)})$ can be regarded as a normed space.
        
        The space of $p$-Sobolev functions in $(\Rp,\m)$ that vanish at infinity is then defined as  the topological closure
        \begin{align*}
            W^{1,p}_0(\Rp,\m):= \overline{\Lip_c(\Rp)\cap W^{1,p}(\Rp,\m)}^{\norm{\cdot}_{W^{1,p}(\Rp,\m)}}
        \end{align*}

\begin{theorem}[P\'olya-Szeg\H{o} inequality, second version]\label{th:v2}
    Let $\Sigma\subseteq\Rnm$ be a $n$-dimensional submanifold such that $\Sigma \in \stc{K}$, for some $K<1/\ic(n)$, and define the measure $\m$ on $\Rp$ as 
    \begin{align*}
        \m_{n,K}:=f_{n,K}\scrL\llcorner\Rp,\quad f_{n,K}(r):=\left(\frac{1}{\ic(n)}-K\right)^n \frac{r^{n-1}}{n^{n-1}}.
    \end{align*}
    If $u\in W^{1,p}_0(\Sigma;\Rp)$ for some $1\leq p<\infty$ and $\u$ is the $\m_{n,K}$-Schwartz rearrangement of $u$ with respect to the measure $\m_{n,K}\in\calM(\Rp)$, then $\u\in W^{1,p}_0(\Rp,\m_{n,K})$ and 
    \begin{align}\label{v2:claim}
        \int |\nabla_{\m_{n,K}} u^{*,\m_{n,K}}|^p\,d\m_{n,K}\leq \int_\Sigma {|\nabla_\Sigma u|^p}.
    \end{align}
\end{theorem}

\begin{definition}
    The space $(\Rp,\m_{n,K})$ of Theorem \ref{th:v2} is called \emph{model space of dimension $n$ and total mean curvature $K$}. 
\end{definition}

Compare this definition with the definition of model space for the curvature-dimension-diameter condition given by Milman in \cite{milman2011sharp}, where the curvature is $0$, the dimension is $n$ and the diameter is $\infty$.

\begin{proof}
    \renewcommand{\u}{{u^{*,\m}}}
    \newcommand{\phim}{{\phi_{*,\m}}}
    \newcommand{\mum}{{\mu_{*,\m}}}
    \newcommand{\psim}{{\psi_{*,\m}}}

    It is enough to prove that
    \begin{align*}
        \int |\nabla u|^p\,d\m_{n,K} \leq\int_\Sigma |\nabla_\Sigma u|^p
    \end{align*}
    for every $u\in \Lip_c(\Sigma;\Rp)$ such that $|\nabla u|\neq 0$ $\Hn$-a.e. in $\{u>0\}$. The general case will follow by density and the definition \eqref{def:Sobm}. Under the standing assumptions, by virtue of Lemma \ref{lemma:techical}, $u^{*,\m_{n,K}}$ is absolutely continuous in $\Rp$. In particular, $\nabla u^{*,\m_{n,K}} = (u^{*,\m_{n,K}})'$ exists and it is negative $\m_{n,K}$-a.e in $\Rp$, which in turn proves that $u^{*,\m_{n,K}}$ is strictly decreasing in $\Rp$.

    To the aim of easing the notation, we write $\m$ instead of $\m_{n,K}$ and $f$ instead of $f_{n,K}$. Moreover, we define $\tau:=(\u)^{-1}$ and set $G(n,K):=n^{1-n}(1/\ic(n)-K)^n$, so that $f(r)=G(n,K)r^{n-1}$.

    We define the function $\phi,\phim,\mu,\mum,\psi,\psim:\Rp\to\Rp$ as
    \begin{gather*}
        \mu(t):=\Hn(\{u>t\}),\quad\mum(t):=\m(\{\u>t\}),\\
        \phi(t):= \int_{\{u>t\}} |\nabla_\Sigma u|,\quad \phim(t):= \int_{\{\u>t\}} |\nabla \u|\,d\m,\\
        \psi(t):=\int_{\{u>t\}}|\nabla_\Sigma u|^p,\quad \psim(t):=\int_{\{\u>t\}}|\nabla \u|^p\,d\m.
    \end{gather*}

    Using the co-area formula, the isoperimetric inequality \eqref{IsoIneq} applied to the sets $\Omega_t:=\{u>t\}$ and H\"{o}lder's inequality with exponents $n$ and $n/(n-1)$, together with our definition of $\m$, $G(n,m,K)$ and Proposition \ref{prop:prop} \ref{prop:norms}, and arguing as in the proof of Theorem \ref{th:PS} we obtain that for a.e. $t>0$
    \begin{align}\label{th2:ineqderivatives}
        -(\phim)'(t)\leq -\phi'(t)\quad \textn{a.e.}\,\,t>0.
    \end{align}
    Therefore, as both $\phi$ and $\phim$ are decreasing and vanishing at $\infty$, \eqref{th2:ineqderivatives} implies
    \begin{align*}
        -\phim(t)\leq \phi(t)\quad \textn{a.e.}\,\,t>0.
    \end{align*}
    Passing to the limit as $t\to0$ in the last inequality, gives the claim when $p=1$.

    Assume $p>1$ and let $t>0$ be such that all of $\phi,\phi_*,\mu,\mu_*,\psi,\psi^*$ are differentiable at $t$ (these $t$'s form full measure subset of $(0,\infty)$). Thanks to H\"{o}lder's inequality, we have
    \begin{align*}
        \frac{\phi(t)-\phi(t+\lambda)}{\lambda}\leq \left(\frac{\psi(t)-\psi(t+\lambda)}{\lambda}\right)^\frac{1}{p}\left(\frac{\mu(t)-\mu(t+\lambda)}{\lambda}\right)^\frac{1}{p'}.
    \end{align*}
    Observe that the assumption $|\nabla_\Sigma u|>0$ $\Hn$-e.e. in $\{u>0\}$ implies $-\mu'(t)>0$ for a.e. $t>0$. Therefore, up to changing the choice of $t$, we can pass to the limit as $\lambda \to 0$ in the above expression. In particular, thanks to the co-area formula, \eqref{th2:ineqderivatives} and the fact that $\mu\equiv \mum$ (by definition of $\m$-Schwartz rearrangement),
    \begin{align}\label{v2:integrate}
        -\psi'(t)\geq \frac{\left(-\phi'(t)\right)^p}{\left(-\mu'(t)\right)^{p-1}}\geq \frac{\left(-(\phim)'(t)\right)^p}{\left(-(\mum)'(t)\right)^{p-1}} = |\nabla \u(\tau(t))|^{p-1}f(\tau(t)).
    \end{align}
    The claim \eqref{v2:claim} is reached after the integration of the inequality \eqref{v2:integrate} from $t=0$ to $t=\infty$.
    
\end{proof}

\nocite{*}

\vspace{1cm}
{\sc Pietro Aldrigo, Universit\"at Bern, Mathematisches Institut (MAI), Sidlerstrasse 12, 3012 Bern, Schweiz}

Email address: pietro.aldrigo@unibe.ch

\vspace{0.5cm}
{\sc Zolt\'an M. Balogh, Universit\"at Bern, Mathematisches Institut (MAI), Sidlerstrasse 12, 3012 Bern, Schweiz}

Email address: zoltan.balogh@unibe.ch

\end{document}